\documentclass[a4paper,11pt]{amsart}
\usepackage{amssymb,latexsym}

\newtheorem{theorem}{Theorem}[section]
\newtheorem{lemma}[theorem]{Lemma}
\newtheorem{corollary}[theorem]{Corollary}
\newtheorem{proposition}[theorem]{Proposition}
\theoremstyle{definition}

\newtheorem{example}[theorem]{Example}

\newtheorem{question}[theorem]{Question}
\theoremstyle{remark}
\newtheorem{remark}[theorem]{Remark}

\numberwithin{equation}{section}

\def \RM{\mathbb{R}}

\def\dr{{\partial_r}}
\def\pa{\partial}
\def\n{\nabla}
\def\nn{\bar\nabla}
\def\cd{\cdot}
\def\f{\phi}
\def\x{\times}

\begin{document}

\title[Sasakian metrics with an additional contact structure]{Sasakian metrics with an additional contact structure}
\author{Tedi Dr\u aghici}
\address{Department of Mathematics \& Statistics\\Florida International University\\Miami, FL 33199}
\email{draghici@fiu.edu}
\author{Philippe Rukimbira}
\address{Department of Mathematics \& Statistics\\Florida International University\\Miami, FL 33199}
\email{rukim@fiu.edu}

\begin{abstract} The question of whether a Sasakian metric
can admit an additional compatible ($K$-)contact structure is addressed.
In the complete case if the second structure is also assumed Sasakian, works of Tachibana-Yu and Tanno
show that the manifold must be 3-Sasakian or an odd dimensional sphere with constant curvature.
Some extensions of this result are obtained, mainly in dimensions 3 and 5.
\end{abstract}

\maketitle

\section{Introduction}
This note is motivated by the following question: {\it On a Sasakian manifold $(M^{2n+1}, g, {\eta})$, is it possible to have a
second $g$-compatible contact structure $(g, \widetilde{\eta})$, with $\widetilde{\eta}
\not\equiv \pm {\eta}$ ?}
\vspace{0.2cm}

\noindent The answer is known to be affirmative; there are examples when $(g, \widetilde{\eta})$ is also Sasakian.
Indeed, if $(M^{2n+1}, g) = (S^{2n+1}, g_0)$ is a standard odd dimensional sphere of constant curvature,
it is well known that the space of Sasakian structures which are compatible with $g_0$
and a given orientation can be identified with the Hermitian symmetric space $ SO(2n+2)/U(n+1)$. Also,
if $(M^{4k+3},g)$ is 3-Sasakian, there is an
$S^2$-family of Sasakian structures compatible with $g$ and the given orientation.
Conversely, by results of Tachibana-Yu and Tanno,
if $(M, g)$ is complete and $(g, \widetilde{\eta})$
is assumed to be Sasakian as well, then the above are essentially the {\it only} examples.

\begin{theorem} \label{tyt} \cite{TY, T} Suppose $(M^{2n+1}, g)$ is a complete Riemannian
manifold which admits two compatible Sasakian structures $(g, {\eta}),
(g, \widetilde{\eta})$, with ${\eta} \not\equiv \pm \widetilde{\eta}$. Then

(a) $(M^{2n+1}, g)$ is covered by the round sphere $(S^{2n+1}, g_0)$, or

(b) $n= 2k+1$, $(M^{4k+3}, g)$ is 3-Sasakian and $\eta, \widetilde{\eta}$ belong to the $S^2$-family
of $g$-compatible Sasakian structures.

Moreover, if the angle function of the two Reeb fields $g(\xi, \widetilde{\xi})$ is non-constant, then only case (a) can occur.
\end{theorem}

\noindent It is thus natural to consider the following refinement of the original question.

\begin{question} \label{q2} Is it possible on a Sasakian manifold
$(M^{2n+1}, g, {\eta})$ to have a second {\bf non-Sasakian} contact
structure $(g, \widetilde{\eta})$?
\end{question}

One may consider local or global versions of the question. One may also allow for
$\eta$ and $\widetilde{\eta}$ to induce the same or opposite orientations on $M^{2n+1}$.
Certainly, this distinction is irrelevant when $dim(M) = 4k+1$. In this case,
if $\widetilde{\eta}$ is a $g$-compatible contact structure, then
$- \widetilde{\eta}$ is also a $g$-compatible contact structure inducing the opposite orientation; thus,
it is enough to treat
Question \ref{q2} in the case when ${\eta}$ and $\widetilde{\eta}$ induce the same orientation.
When $dim(M) = 4k+3$ the two cases can be quite different.
For instance, Question \ref{q2} has a negative answer for a compact Sasakian
3-manifold $(M^3, g, {\eta})$,  if the second contact structure $\widetilde{\eta}$ induces the same orientation
(see Corollary \ref{compact3d-so}). However, when ${\eta}$ and $\widetilde{\eta}$
induce opposite orientations, $\widetilde{SL}_2$ and its compact quotients provide affirmative examples for
Question \ref{q2}. At this time, we do not know if these are the only such 3-dimensional examples. We prove
that locally they are the only examples with constant scalar curvature.
\begin{theorem} \label{dim3sconst}
Let $(M^3, g)$ be a connected Riemannian 3-manifold with constant scalar curvature.
Assume that the metric $g$ admits a compatible Sasakian structure $(g, {\eta})$ and another
compatible contact structure $(g, \widetilde{\eta})$ with $\widetilde{\eta} \not\equiv \pm {\eta}$.
One of the following holds:

(i) The structure $(g,\widetilde{\eta})$ is also Sasakian and, in this case,
 $(M^3,g)$ is locally isometric to the standard sphere $(S^3, g_0)$.

(ii) The structure $(g, \widetilde{\eta})$ is not Sasakian; in this case ${\eta}, \widetilde{\eta}$ must induce opposite orientations,
 and $(M^3,g)$ must be locally isometric to $(\widetilde{SL}_2, g_0)$ as described in Example \ref{basicexp}.
\end{theorem}

Contrasting with case (ii) of Theorem \ref{dim3sconst}, note that Question \ref{q2} has a negative answer
in all dimensions for complete Einstein manifolds (see Theorem 1.2 in \cite{adm}).
This is a consequence of a theorem of Boyer and Galicki \cite{bg} that any complete $K$-contact Einstein
manifold must be Sasakian and a result of Blair characterizing $K$-contact structures by the Ricci
curvature in the direction of the Reeb field (\cite{Blair}, Theorem 7.1).

\vspace{0.2cm}

In section 4, we treat the special case of Question \ref{q2} when we assume that
$(g, \widetilde{\eta})$ is a $K$-contact structure, i.e. the Reeb vector field $ \widetilde{\xi}$ of
$\widetilde{\eta}$ is a Killing vector field for the metric $g$.
In this case, if the Reeb vector fields do not commute, i.e. $[{\xi}, \widetilde{\xi}] \not\equiv 0$,
Theorem \ref{tyt} easily implies that odd-dimensional spheres and 3-Sasakian manifolds are the only possible complete examples.
Thus, the Sasakian--$K$-contact version of Question \ref{q2} has a negative answer
if the Reeb fields do not commute (see Proposition \ref{SKnotcommuting}).
Even more generally, Belgun, Moroianu, Semmelmann \cite{BMS} obtained a beautiful characterization
of $K$-contact manifolds
$(M^{2n+1}, g, \xi_1)$ which admit another unit Killing field $ \xi_2 $, orthogonal to $\xi_1$, and with $[{\xi}_1, {\xi}_2] = \xi_3 \not\equiv 0$.
They show that such a manifold must have dimension $4k+3$ and $(M^{4k+3}, g, \xi_1, \xi_2, \xi_3)$
is a so-called {\it weakly $K$-contact 3-structure}. This means that $\xi_1, \xi_2, \xi_3$ are orthogonal,
all induce $K$-contact structures with respect to the given metric $g$, and there is a splitting
$TM = D_+ \oplus D_- \oplus V \; ,$ where $V={\rm Span}(\xi_1, \xi_2, \xi_3)$, so that the corresponding
endomorphisms $\phi_1, \phi_2, \phi_3$ satisfy the quaternionic relations on $D_-$ and the anti-quaternionic relations
on $D_+$.
See \cite{BMS} for more details on weakly $K$-contact 3-structures. It is also shown there how these structures arise
naturally from pseudo-Riemannian 3-Sasakian manifolds.

As it is related to our main question, let us also mention here a result of Kashiwada \cite{K}
that a contact metric 3-structure $(g, \xi_1, \xi_2, \xi_3)$ (definition as above without
$K$-contact assumption and with $D_+$ trivial) is necessarily 3-Sasakian.

\vspace{0.2cm}

For our Sasakian--$K$-contact version of Question \ref{q2}, it thus remains
to consider the case when the Reeb vectors commute. We provide the following partial answer.
\begin{theorem} \label{codim2}
Let $(M^{2n+1}, g, {\eta}, \xi)$ be a compact Sasakian manifold
which admits another $g$-compatible $K$-contact structure $(g, \widetilde{\eta}, \widetilde{\xi})$
with $[\xi, \widetilde{\xi}] = 0$.
Let $f = g(\xi, \widetilde{\xi})$ be the angle function of the two Reeb fields and assume
that one of its critical sub-manifolds $f^{-1}(1)$, or $f^{-1}(-1)$ has codimension 2.
Then $(g, \widetilde{\eta}, \widetilde{\xi})$  is also Sasakian and the manifold is covered by a sphere.
\end{theorem}
\noindent In dimension 5, we have the following complete answer for the Sasakian--$K$-contact version of Question \ref{q2}.
\begin{theorem} \label{dim5} Suppose $(M^5, g, {\eta},\xi )$ is a complete Sasakian 5-manifold
which admits another $g$-compatible $K$-contact structure $(g, \widetilde{\eta}, \widetilde{\xi})$.
Then $(g, \widetilde{\eta})$ is also Sasakian and if $\xi \not\equiv \pm \widetilde{\xi}$, then
$(M^5, g)$ is covered by the round sphere $(S^5, g_0)$.
\end{theorem}

\noindent For dimensions higher than 5 the Sasakian--$K$-contact problem still remains open
in the case of commuting Reeb fields.

\vspace{0.2cm}

A crucial role throughout this note is played by the angle function $f = g(\xi, \widetilde{\xi})$.
The study of this angle function goes back to the work of Professor Banyaga with the second author \cite{BAR}.
Some ideas from \cite{BAR} and \cite{RUK} are essential in the proof of Theorem \ref{codim2}.
Thus, the authors find it particularly fitting to dedicate this work to Professor Banyaga.



\section{Preliminaries}

A contact form on an orientable odd dimensional manifold $M^{2n+1}$ is a 1-form $\eta$ so that
$\eta \wedge (d\eta)^n$ is a volume form on $M^{2n+1}$, i.e. it does not vanish at any point
of the manifold. It is well known that each contact form $\eta$ has a uniquely associated Reeb field
$\xi$, so that
\begin{equation} \label{reebdef}
 \eta(\xi) = 1 \; , \; \; i_{\xi} d \eta = 0 \; ,
\end{equation}
where $i_X$ denotes the interior product of a form with a tangent vector $X \in TM$. As an immediate consequence
of its definition, the Reeb field satisfies
\begin{equation} \label{Lxieta}
L_{\xi} \eta = 0 \; , \; \; \; L_{\xi} d\eta = 0 \; ,
\end{equation}
where $L_X$ is the Lie derivative with respect to $X \in TM$.
We denote by $\mathcal{D} = {\rm ker} \, \eta$ the contact distribution associated to $\eta$.

A Riemannian metric $g$ is said to be {\it compatible} (or {\it associated}) to $\eta$ if the skew-symmetric endomorphism
$\phi$ of $TM$ defined by
\begin{equation} \label{phidef}
  \frac{1}{2} d \eta (X, Y) = g (X, \phi Y)
\footnote{We generally follow the notations and conventions from \cite{Blair}.
One exception is the definition of the differential $d$. For us the differential of a 1-form $\alpha$ is
$d \alpha(X,Y) := X(\alpha(Y)) - Y(\alpha(X)) - \alpha([X,Y])$. In \cite{Blair}, the right hand-side has a ``1/2''
factor.}
\end{equation}
is an almost complex structure on the contact distribution $\mathcal{D}$. In view of (\ref{reebdef}), this is equivalent to the condition
\begin{equation} \label{phisq}
 \phi^2 = -Id + \eta \otimes \xi \; \; .
\end{equation}
The tuple $(M^{2n+1}, g, \phi, \eta, \xi)$ as above is called a {\it contact metric structure} and, as already done in the introduction,
we sometimes refer to it just as $(M^{2n+1}, g, \eta)$.

A consequence of (\ref{phidef}) and (\ref{phisq}) is that the volume form $dV_g$, defined by the metric $g$ and the orientation
induced by $\eta$, coincides with the volume
form induced by $\eta \wedge (d \eta)^n$, up to a dimension factor:
\begin{equation} \label{volforms}
dV_g = \frac{1}{2^n n !} \eta \wedge (d \eta)^n \; .
\end{equation}

Recall that a contact metric structure $(g, \phi, \eta, \xi)$ is said to be $K$-{\it contact} if the Reeb vector field is a Killing vector field
for the metric, i.e. $L_{\xi} g = 0$. A contact metric structure $(g, \phi, \eta, \xi)$
is called {\it Sasakian} if the covariant derivative of $\phi$ with respect to the Levi-Civita connection is given by
\begin{equation} \label{phiSasaki}
({\nabla_X \phi}) Y =  g(X,Y)\xi - \eta(Y) X \; , \; \; \forall \; X,Y \in TM \; .
\end{equation}
It is well known that the Sasakian condition always implies the $K$-contact condition, with the converse being true
only in dimension $3$.


\vspace{0.2cm}

Let $(M^{2n+1}, g, \phi, \eta, \xi)$ be a contact metric structure. Following \cite{Blair}, define
the tensor field $h$ by
$$ h = \frac{1}{2} L_{\xi} \phi \; \; .$$
It is easy to see that $h$ is $g$-symmetric, $\phi$-anti-invariant and $h \xi = 0$.
Also, the tensor $h$ can be seen as the obstruction to the structure being $K$-contact,
as $h \equiv 0$ if and only if $L_{\xi} g \equiv 0$.


%

In the following proposition we collect some basic facts about contact metric structures.

%
%
%
%
%
%
%
%

\begin{proposition} \label{curvxi}
Let $(M^{2n+1}, g, \phi, \eta, \xi)$ be a contact metric structure.
Denote by $R$, $Ric$, $s$ respectively the Riemannian curvature tensor,
the Ricci tensor and the scalar curvature of the metric $g$. Let $\delta^g$ be the divergence operator induced by the metric
and also denote by $l$ the symmetric endomorphism of $TM$ defined by $lX = R(\xi, X) \xi$. The following relations hold:
\begin{equation} \label{nablaxiphi}
\nabla_{\xi} \phi = 0 \; ;
\end{equation}
\begin{equation} \label{nablaXxi}
\nabla_X \xi = - \phi X - \phi h (X) \; ;
\end{equation}
\begin{equation} \label{deltaeta}
\delta^g \eta = 0 \; ,  \; \; \; \delta^g(d \eta)  = 4n \, \eta
\end{equation}
\begin{equation} \label{nablaxiphih}
(\nabla_{\xi} (\phi h)) X = \phi^2 X + (\phi h)^2 X - l X \; ;
\end{equation}
\begin{equation} \label{ricxi}
Ric(\xi, \cdot) = 2n \, \eta + \delta^g (\phi h) \; ;
\end{equation}
\begin{equation} \label{ricxixi}
Ric(\xi,\xi) = 2n - |h|^2 \; ;
\end{equation}
\begin{equation} \label{xihsq}
\xi(|h|^2) = - 2 <l, \phi h> \; ;
\end{equation}
\begin{equation} \label{deltaricxi}
 - \xi(|h|^2) + \delta^g(Ric(\xi, \cdot)^{\mathcal{D}})
= -\frac{1}{2} \xi(s) + < Ric, \phi h > \; .
\end{equation}
\end{proposition}

\begin{proof} All relations, except perhaps the last two, are well known. For proofs see Blair's book \cite{Blair}.
Relation (\ref{xihsq}) follows by taking the inner product of both sides of (\ref{nablaxiphih}) with $\phi h$,
then use $\phi h \xi = 0$, $tr(\phi h) = 0$ and $<(\phi h)^2, \phi h> = 0$ (as $\phi h$ is $\phi$-anti-invariant and $(\phi h)^2$ is $\phi$-invariant).

Relation (\ref{deltaricxi}) follows from alternative ways of computing $ \delta^g(Ric(\xi, \cdot))$.  Indeed, from (\ref{ricxi}) and (\ref{ricxixi})
we have
$$ Ric(\xi, \cdot) = (2n - |h|^2) \, \eta + (Ric(\xi, \cdot))^{ \mathcal{D}} \; ,$$
and taking divergence of this relation we get
$$ \delta^g(Ric(\xi, \cdot)) = - \xi(|h|^2) + \delta^g(Ric(\xi, \cdot)^{\mathcal{D}}) \; .$$
Computing $ \delta^g(Ric(\xi, \cdot))$ directly by making use of (\ref{nablaXxi}), we have
$$ \delta^g(Ric(\xi, \cdot)) = (\delta^g Ric) (\xi) + <Ric \, , \, \phi h> = - \frac{1}{2} \xi(s) + < Ric, \phi h > \; , $$
where for the last equality we used the Bianchi differential identity $\delta^g Ric = - 1/2 ds$.
\end{proof}

As a consequence of independent interest of the relation (\ref{deltaricxi}), note the following
\begin{corollary}
Let $(M^{2n+1}, g, \phi, \eta, \xi)$ be a contact metric structure. Suppose that the Reeb field is an eigenvector
of the Ricci tensor and that the Ricci tensor is $\phi$-invariant.
Then $\xi(s - 2|h|^2) = 0$.
\end{corollary}


Next consider a Riemannian manifold $(M^{2n+1}, g)$ which admits two compatible contact structures
$\eta$, $\widetilde{\eta}$. Let $\xi, \widetilde{\xi}$ be the two Reeb fields and let $\phi$, $\widetilde{\phi}$ be the
almost complex structures on the corresponding contact bundles. In the next proposition we make some first general observations
on the angle function $f := g({\xi},\widetilde{\xi})$ of the two Reeb fields.

\begin{proposition} \label{laplangle}
Suppose $(M^{2n+1}, g)$ is a Riemannian manifold which admits two compatible contact structures $\eta$, $\widetilde{\eta}$
 and let $f = g({\xi},\widetilde{\xi})$ denote the angle function. Then

\vspace{0.2cm}

(a) ${\eta} \wedge (d {\eta})^n = \pm \widetilde{\eta} \wedge (d \widetilde{\eta})^n$ and ${\eta}(\widetilde{\xi}) = \widetilde{\eta}({\xi}) = f$,
where the $\pm$ sign in the first relation corresponds to whether ${\eta}$ and $\widetilde{\eta}$ induce the same, or opposite orientations on $M$.

\vspace{0.2cm}

(b) The gradient of the angle function is given by
\begin{equation} \label{gradfgen}
\nabla f =  {\phi} \widetilde{\xi} + \widetilde{\phi} {\xi} - {\phi} {h} \widetilde{\xi} - \widetilde{\phi} \widetilde{h} {\xi} \; .
\end{equation}

(c) The Laplaceian of the angle function is given by
\begin{equation} \label{lapf}
\Delta f = 8n f - 2<{\phi},\widetilde{\phi}> - 2<{\phi} {h},\widetilde{\phi} \widetilde{h}> - 2Ric({\xi},\widetilde{\xi}) \; .
\end{equation}
In particular, if one of the structures is $K$-contact, then
\begin{equation} \label{lapfK}
\Delta f = 4n f - 2<{\phi},\widetilde{\phi}> \; .
\end{equation}
\end{proposition}

\begin{proof} Part (a) is obvious and part (b) follows directly from relation (\ref{nablaXxi}) of Proposition \ref{curvxi}.
For the Laplaceian of $f$, we have
$$ \Delta f = \Delta <\eta, \widetilde{\eta}> =
<\nabla^* \nabla \eta, \widetilde{\eta}> + <\eta, \nabla^* \nabla \widetilde{\eta}> - 2 <\nabla \eta, \nabla \widetilde{\eta}> \; .$$
Now we use the Weitzenb\"ock formula for 1-forms applied to ${\eta}$ and $\widetilde{\eta}$.
$$ (d \delta^g + \delta^g d) \eta = \nabla^* \nabla \eta + Ric(\xi, \cdot) \; . $$
By relation (\ref{deltaeta}), $ (d \delta^g + \delta^g d) \eta = 4n \; \eta $, thus
$$ \nabla^* \nabla \eta = 4n \, \eta - Ric(\xi, \cdot) \; ,$$
and the similar formula for $ \nabla^* \nabla \widetilde{\eta}$.
Now formula (\ref{lapf}) follows noting that
\newline $<\nabla \eta, \nabla \widetilde{\eta}> = <\nabla \xi, \nabla \widetilde{\xi}>$ and using (\ref{nablaXxi}).
Relation (\ref{lapfK}) follows from  formula (\ref{ricxi}) specialized to the $K$-contact case.
\end{proof}

\section{The 3-dimensional case}

In this section we are interested in the 3-dimensional version of Question \ref{q2} and we will prove Theorem \ref{dim3sconst}.
%
We start with the following example of contact forms sharing a compatible metric.
The example is inspired from \cite{GG}, where it appeared in the context of classifying
taut-contact circles. See also \cite{M} for the classification of 3-dimensional Lie groups.

\begin{example} \label{basicexp}
{\rm Let $\mathcal{G}$ denote one of the simply connected Lie groups $S^3=SU(2), \widetilde{SL}_2 , \widetilde{E}_2$, where
$\widetilde{SL}_2$ is the universal cover of $PSL_2\mathbb{R}$ and $\widetilde{E}_2$ is the universal cover of the Euclidean
group (i.e. orientation preserving isometries of $ \mathbb{R}^2$).
Then the Lie algebra
of $\mathcal{G}$ admits a basis $\xi_1, \xi_2, \xi_3$ satisfying the structure equations
$$ [\xi_2, \xi_3] = 2 \xi_1, [\xi_3, \xi_1] = 2 \xi_2, [\xi_1, \xi_2] = 2 \epsilon \xi_3, $$
where $\epsilon = 1, -1$, or $0$ for each of the three groups listed, respectively. Denote by $\eta_1, \eta_2, \eta_3$ the co-frame
dual to $\xi_1, \xi_2, \xi_3$ and let the Riemannian metric
$$g_0 =  \eta_1 \otimes \eta_1 + \eta_2 \otimes \eta_2 + \eta_3 \otimes \eta_3 . $$
In the case of $SU(2)$, $(\eta_1, \eta_2, \eta_3)$ are all contact structures with
the same volume form and are all compatible with the metric $g_0$; in fact,
$(g_0, \eta_1, \eta_2, \eta_3)$ is the standard 3-Sasakian structure
on the 3-dimensional sphere $S^3 = SU(2)$. In the case of $ \widetilde{SL}_2 \mathbb{R}$, or $ \widetilde{E}_2$,
$\eta_1$ and $\eta_2$ are contact structures with the same volume form, both compatible with $g_0$.
In fact, any contact form $\widetilde{\eta} = a_1 \eta_1 + a_2 \eta_2$ of the (taut contact) circle determined by
$\eta_1, \eta_2$ ($a_1^2 + a_2^2 = 1$) is compatible with $g_0$ and is non-Sasakian.
In the case of $ \widetilde{SL}_2 \mathbb{R}$, $(g_0, \eta=\eta_3)$ is a Sasakian structure inducing the opposite orientation
as that of $(g_0, \widetilde{\eta})$.}
\end{example}

We next give a lemma needed in the proof of Theorem \ref{dim3sconst}, but which has some independent
interest.

%
%

\begin{lemma} \label{commReebs} Assume that ${\eta}$, $\widetilde{\eta}$ are two contact forms on (an open set of) a 3-dimensional
manifold, with Reeb fields ${\xi}$ and $\widetilde{\xi}$, respectively. Assume also that at all points
$$ {\eta} \wedge d {\eta} = \widetilde{\eta} \wedge d \widetilde{\eta}, \; \; \; {\eta}(\widetilde{\xi}) = \widetilde{\eta}({\xi}), \; \; \; \mbox{ and} \; \;
[{\xi}, \widetilde{\xi}] = 0 \; .$$
Then, up to sign, the Reeb fields must coincide, i.e. ${\xi} \equiv \pm \widetilde{\xi}$. Moreover, $\widetilde{\eta} = \pm {\eta} + a$, where
$a$ is a closed, basic 1-form (i.e. $d a = 0$, $a(\xi) = 0$).
In particular, such contact forms ${\eta}, \widetilde{\eta}$ admit a common compatible metric if only if ${\eta} \equiv \pm \widetilde{\eta}$, i.e. $a \equiv 0$.
\end{lemma}

\begin{proof}
Applying $ {\eta} \wedge d {\eta} = \widetilde{\eta} \wedge d \widetilde{\eta}$ to $(\xi, \widetilde{\xi}, X)$ and using
${\eta}(\widetilde{\xi}) = \widetilde{\eta}({\xi})$, we get
$$i_{ \widetilde{\xi}} d \eta = - i_{\xi} d \widetilde{\eta} \; .$$
From the assumption $[{\xi}, \widetilde{\xi}] = 0$, we also have $ L_{ \widetilde{\xi}} \eta = L_{\xi} \widetilde{\eta} = 0$
and, using Cartan's formula, it follows that
$$ d({\eta}(\widetilde{\xi})) = 0, \; \; i_{ \widetilde{\xi}} d \eta = 0, \; \;   i_{\xi} d \widetilde{\eta} = 0 \; \; .$$
These relations imply that $\xi = \pm \widetilde{\xi}$. Assume $\xi = \widetilde{\xi}$, the other case being similar.
Let $a:= \widetilde{\eta} - \eta$. Clearly, $a(\xi) = 0$, so $a \wedge d \eta = a \wedge d \widetilde{\eta} = 0$.
From
$$ \widetilde{\eta} \wedge d \widetilde{\eta} = \eta \wedge d \eta + \eta \wedge da + a \wedge d \widetilde{\eta} \; , $$
combined with $ {\eta} \wedge d {\eta} = \widetilde{\eta} \wedge d \widetilde{\eta}$, it follows
that $\eta \wedge da = 0$. But we also have  $i_{\xi} da = 0$, so these imply $da = 0$, i.e. $a$ is a closed, basic 1-form,
as stated. Finally, note that the forms $\widetilde{\eta} = {\eta} + a$ and $\eta$ share a common compatible metric $g$
if and only if $a =0$, as both forms must be the $g$-duals of the common Reeb field $\xi$.
\end{proof}

\begin{remark} {\rm As a consequence of Lemma \ref{commReebs}, observe that the conditions
$$ {\eta} \wedge d {\eta} = \widetilde{\eta} \wedge d \widetilde{\eta}, \; \; \; {\eta}(\widetilde{\xi}) = \widetilde{\eta}({\xi}) \; ,$$
although necessary, are {\bf not} also sufficient for insuring that two contact forms, like-wise oriented,
$\eta, \widetilde{\eta}$ admit a common compatible metric. They are sufficient on the set of points where
$f = {\eta}(\widetilde{\xi}) = \widetilde{\eta}({\xi})$ is not equal to $\pm 1$.}
\end{remark}

Next we specialize Proposition \ref{laplangle} to the 3-dimensional problem.


\begin{proposition} \label{angle3d}
Let $(M^3, g)$ be a 3-dimensional Riemannian manifold which admits a compatible Sasakian structure
$(g, {\phi}, {\xi}, {\eta})$ and another compatible contact structure
$(g, \widetilde{\phi}, \widetilde{\xi}, \widetilde{\eta})$. As before, let $f = g({\xi},\widetilde{\xi})$ denote the angle function between the two Reeb fields.
Then we have the following:

(a) If ${\eta}$ and $\widetilde{\eta}$ induce the same orientation on $M^3$, then $ {\phi} \widetilde{\xi} = - \widetilde{\phi} {\xi}$.
Moreover, in this case the gradient and the Laplaceian of the angle function satisfy
\begin{equation} \label{lapf3d-so}
\nabla f = -\widetilde{\phi} \widetilde{h} {\xi} \; , \; \; \; \; \Delta f = 0 \; .
\end{equation}

(b) If ${\eta}$ and $\widetilde{\eta}$ induce opposite orientations on $M^3$, then $ {\phi} \widetilde{\xi} = \widetilde{\phi} {\xi}$.
Moreover, in this case
\begin{equation} \label{lapf3d-oo}
\nabla f = 2\widetilde{\phi} {\xi} -\widetilde{\phi} \widetilde{h} {\xi} \; , \; \; \; \; \Delta f = 8 f \; .
\end{equation}

Note that in either case $f$ satisfies an elliptic equation, so it has the (strong) unique continuation property.
Also note that in either case we have that $\widetilde{\xi}(f) = 0$.
\end{proposition}

\begin{proof} As in the proof of Lemma \ref{commReebs}, condition $ {\eta} \wedge d {\eta} = \pm \widetilde{\eta} \wedge d \widetilde{\eta}$
implies
$$i_{ \widetilde{\xi}} d \eta = \mp i_{\xi} d \widetilde{\eta} \; , \; \mbox{ hence }  {\phi} \widetilde{\xi} = \mp \widetilde{\phi} {\xi},$$
where the signs correspond. The formulae for $\nabla f$ now follow from (\ref{gradfgen}). The formulae for $\Delta f$ follow from (\ref{lapfK})
noting that in dimension 3
$$ <\phi, \widetilde{\phi}> = \frac{1}{2} < d \eta,  d \widetilde{\eta}> = \frac{1}{2} < \star_g (d \eta),  \star_g (d \widetilde{\eta})>
= \pm 2 < \eta,  \widetilde{\eta}> = \pm 2f ,$$
where $\star_g$ is the Hodge operator of the metric $g$.
\end{proof}

The proof of Theorem \ref{dim3sconst} follows from Propositions \ref{fconst}, \ref{key3dprop}, \ref{local3dtyt} below.
The first one deals with the case of constant angle function.

\begin{proposition} \label{fconst} Let $(M^3, g)$ be a 3-dimensional Riemannian manifold which admits a compatible Sasakian structure
$(g, {\phi}, {\xi}, {\eta})$ and another compatible contact structure
$(g, \widetilde{\phi}, \widetilde{\xi}, \widetilde{\eta})$. Suppose that the angle function $f = g({\xi},\widetilde{\xi})$ is constant.
Then either $(g, \widetilde{\phi}, \widetilde{\xi}, \widetilde{\eta})$ is Sasakian as well,
or $(M^3,g)$ is locally isometric with $(\widetilde{SL}_2, g_0)$ as presented in Example \ref{basicexp}.
\end{proposition}

\begin{proof} In the case $\eta, \widetilde{\eta}$ induce the same orientation, the gradient of the angle function formula (\ref{lapf3d-so})
shows that $\nabla f = 0$ implies $\widetilde{\phi} \widetilde{h} = 0$, i.e. $(g, \widetilde{\phi}, \widetilde{\xi}, \widetilde{\eta})$ is Sasakian.

 In the case $\eta, \widetilde{\eta}$ induce opposite orientations, the constant angle function $f$ must be identically 0, by
 the Laplaceian relation in (\ref{lapf3d-oo}). Thus the vectors $\xi, \widetilde{\xi}, \widetilde{\phi} \xi$ form a $g$-orthonormal basis.
 From the gradient relation in (\ref{lapf3d-oo}), we get that $\widetilde{\phi} \xi$ is an eigenvector of $ \widetilde{h}$ with eigenvalue $-2$.
 As $ \widetilde{h}$ anti-commutes with $ \widetilde{\phi}$, $ \xi$ is also an eigenvector of $ \widetilde{h}$   with eigenvalue $2$.
Using relation (\ref{nablaXxi}) for $(g, \widetilde{\eta})$ and the Sasakian condition for $(g, \eta)$ it is easy to check
the following Lie brackets relations:
$$ [\xi, \widetilde{\xi}] =  -2 \widetilde{\phi} \xi, \; \;
[  \widetilde{\xi},  \widetilde{\phi}\xi ] =  2 \xi ,
\; \; [ \widetilde{\phi}\xi , \xi ] =  2 \widetilde{\xi} \; .  $$
Thus our manifold is locally isometric with  $(\widetilde{SL}_2, g_0)$ as presented in Example \ref{basicexp}.
\end{proof}

An immediate consequence of Propositions \ref{angle3d} and \ref{fconst} is:

\begin{corollary} \label{compact3d-so} Let $(M^3, g)$ be a compact 3-dimensional Riemannian manifold which admits two compatible contact
structures $(g, {\phi}, {\xi}, {\eta})$, $(g, \widetilde{\phi}, \widetilde{\xi}, \widetilde{\eta})$
inducing the same orientation. If one of the structures is Sasakian then the other is Sasakian as well,
and if $\eta \not\equiv \pm \widetilde{\eta}$, $(M^3, g)$ must be locally isometric with $(S^3, g_0)$.
\end{corollary}
\begin{proof} From relation (\ref{lapf3d-so}), the angle function is harmonic, so it must be constant
as the manifold is assumed compact. Then the conclusion follows from Proposition \ref{fconst}.
\end{proof}

We can draw conclusions without compactness and irrespective of orientation if we assume
the constancy of the scalar curvature.
\begin{proposition} \label{key3dprop}
Let $(M^3, g)$ be a 3-dimensional Riemannian manifold which admits a compatible Sasakian structure
$(g, {\phi}, {\xi}, {\eta})$ and another compatible contact structure
$(g, \widetilde{\phi}, \widetilde{\xi}, \widetilde{\eta})$. If the scalar curvature is constant, then
the angle function $f = g({\xi},\widetilde{\xi})$ must be constant, or the second structure is Sasakian as well.
\end{proposition}

\begin{proof} We assume $f$ is not identically $1$, or $-1$ as otherwise there is nothing to prove.
As $f$ satisfies the unique continuation property, the set of points where $f^2 \neq 1$ is an open dense
set. We'll work on this set and then extend the conclusions by continuity.

It is well known that in dimension 3, the whole curvature tensor is determined by the Ricci tensor.
Using this, relation (\ref{xihsq}) for the structure $(g, \widetilde{\phi}, \widetilde{\xi}, \widetilde{\eta})$ takes the form
$$ \widetilde{\xi} (|\widetilde{h}|^2) = 2<Ric, \widetilde{\phi} \widetilde{h}> \; .$$
On the other hand, the assumption that $(M^3, g, {\eta})$ is Sasakian implies that
the Ricci tensor is given by
$$ Ric = \Big( \frac{s}{2} - 1 \Big) g + \Big( 3 - \frac{s}{2} \Big) {\eta} \otimes {\eta} \; . $$
From these two relations we get
\begin{equation} \label{key3drel}
\widetilde{\xi} (|\widetilde{h}|^2) = (6 - s) <\widetilde{\phi} \widetilde{h} \xi \;, \; \xi > \; .
\end{equation}
From the above expression of the Ricci tensor we also get
$$Ric(\widetilde{\xi},\widetilde{\xi})  = \Big( \frac{s}{2} - 1 \Big) + \Big( 3 - \frac{s}{2} \Big) f^2 \; .$$
Combining this with relation (\ref{ricxixi}), we obtain
\begin{equation} \label{sfh}
 (1 - f^2) ( 3 - \frac{s}{2} ) = | \widetilde{h} |^2  \; .
\end{equation}
Noting also that $ \widetilde{\xi}(f) = 0$ and $\xi(f) = - <\widetilde{\phi} \widetilde{h} \xi \;, \; \xi >$
(in either case when $\eta,  \widetilde{\eta}$ induce the same or opposite orientation),
we obtain from (\ref{key3drel}) and (\ref{sfh})
\begin{equation} \label{nablasf}
(1-f^2)\widetilde{\xi} (6 -s) = - 2 (6 - s) \xi (f) \; .
\end{equation}

Using the assumption that the scalar curvature is constant and the fact that $ \widetilde{\xi}(f) = 0$,
it follows that both sides of relation (\ref{nablasf}) vanish.
From (\ref{sfh}), $s = 6$ if and only if $ \widetilde{h}= 0$. It remains to consider the case
$ \xi(f) = - <\widetilde{\phi} \widetilde{h} \xi \;, \; \xi > = 0$. Note that this condition is equivalent
with $ \widetilde{\phi} \xi$ being an eigenvector for $ \widetilde{h}$. Let $\lambda$ be the corresponding
eigenvalue. From now on we split the cases when ${\eta}$ and $\widetilde{\eta}$
induce the same or opposite orientation, although the arguments are very similar.

In the case of the same orientation we have
$$ \nabla f = - \widetilde{\phi} \widetilde{h} {\xi} = \lambda  \widetilde{\phi} {\xi} \; ,$$
$$[{\xi}, \widetilde{\xi}] = \nabla_{{\xi}} \widetilde{\xi} -  \nabla_{\widetilde{\xi}} {\xi} =
- 2\widetilde{\phi} {\xi} - \widetilde{\phi} \widetilde{h} {\xi} = (- 2 + \lambda)  \widetilde{\phi} {\xi}.$$
As $ \widetilde{\xi}(f) = \xi(f) = 0$, we have $ [{\xi}, \widetilde{\xi}] (f) = 0$. Thus the two relations above
imply
$$ 0 = [{\xi}, \widetilde{\xi}] (f) = g( [{\xi}, \widetilde{\xi}] , \nabla f ) = \lambda (-2 + \lambda) (1-f^2) .$$
As we work on the (open, dense) set where $f^2 \neq 1$, it follows that $\lambda = 0$ or $\lambda = 2$.
But $\lambda = 2$ corresponds to $ [{\xi}, \widetilde{\xi}] = 0$, a case which cannot occur, by Lemma \ref{commReebs}
(as $\widetilde{\eta} \neq \pm \eta$). Thus $\lambda = 0$, i.e. $\nabla f = 0$ is the only possibility in this case.

In the case of opposite orientation we have
$$ \nabla f = 2\widetilde{\phi} {\xi} - \widetilde{\phi} \widetilde{h} {\xi} = (2 + \lambda)  \widetilde{\phi} {\xi}  \; ,$$
$$[{\xi}, \widetilde{\xi}] = \nabla_{{\xi}} \widetilde{\xi} -  \nabla_{\widetilde{\xi}} {\xi} =
- \widetilde{\phi} \widetilde{h} {\xi} = \lambda  \widetilde{\phi} {\xi} \; .$$
As before,
$$ 0 = [{\xi}, \widetilde{\xi}] (f) = g( [{\xi}, \widetilde{\xi}] , \nabla f ) = \lambda (2 + \lambda) (1-f^2) ,$$
so $\lambda = 0$ or $\lambda = -2$.
In this case, $\lambda = 0$, corresponds to $ [{\xi}, \widetilde{\xi}] = 0$, but also implies that $\widetilde{h} = 0$,
that is the structure $(g, \widetilde{\eta})$ is Sasakian. The case $\lambda = -2$ corresponds to $\nabla f = 0$.
\end{proof}

The final step of the proof of Theorem \ref{dim3sconst} is the elementary
observation that in dimension 3, Theorem \ref{tyt} has the following local version.
\begin{proposition} \label{local3dtyt}
Suppose $(M^3, g)$ is a connected Riemannian
manifold which admits two compatible Sasakian structures $(g, {\eta}),
(g, \widetilde{\eta})$, with ${\eta} \not\equiv \pm \widetilde{\eta}$.
Then $(M^3, g)$ is locally isometric to the round sphere $(S^{3}, g_0)$.
\end{proposition}
\begin{proof}
From the assumption ${\eta} \not\equiv \pm \widetilde{\eta}$ and Corollary \ref{angle3d}, it follows
that the set of points where the angle function $f = g({\xi}, \widetilde{\xi})$ is not equal to $\pm 1$ is dense
in $M$. On this set, using the Sasakian condition for both ${\xi}$, $\widetilde{\xi}$, it follows that
$Ric = 2g$. In dimension 3, this implies that the metric has constant
sectional curvature $1$ and, by density, this is true on the whole $M$. Thus, $g$ is locally isometric to the standard
round sphere metric $g_0$.
\end{proof}

This concludes the proof of Theorem \ref{dim3sconst}. Certainly it is natural to ask what happens
 if the constant scalar curvature condition is removed. At this time we don't know the answer
 and we leave this as an open question.

\section{Proof of Theorems \ref{codim2} and \ref{dim5}}
This section has a better presentation at the level of the symplectic cone,
so we start with some brief preliminaries on this. It is well known that there is a
dictionary between a contact manifold $(M^{2n+1}, \eta)$ and its symplectic cone
$(\bar M:=M\x\RM^*_+ , \omega)$, with the symplectic form given by
\begin{equation}\label{om}
\omega =r dr\wedge\eta+\frac{r^2}{2}d\eta \; ,
\end{equation}
where $r$ denotes the coordinate on the $\RM^*_+$-factor. The dictionary extends to compatible metrics
as well, see \cite{BG}. If $(g, \phi, \eta, \xi)$ is a contact metric structure on $M$ this corresponds
to an almost K\"ahler structure $(\bar g, J, \omega)$ on $ \bar M$ where $\omega$ is the form defined above,
the cone metric is
$$ \bar g = dr^2 + r^2g \;  ,$$
and the almost complex structure is given by
$$ J = \phi \mbox{ on } \mathcal{D} = {\rm Ker} \; \eta \; , \mbox{ and } J \xi = r \dr \; ,$$
where $\dr$ denotes the vector field $\frac{\pa}{\pa r}$ on $\bar M$, and $r \dr$ is the so called Liouville vector field
on the cone.

We record below (see e.g. \cite{O'N}) formulas relating the Levi-Civita
connections $\n$ and $\nn$ of $(M,g)$ and $(\bar M, \bar g)$.
\begin{equation}\label{vec}
\nn_\dr\dr=0; \; \; \; \nn_X\dr=\nn_\dr X=\frac{1}{r}X; \; \; \;
\nn_XY=\n_XY-rg(X,Y)\dr.
\end{equation}
Using this, we obtain for every vector $X$ and a $p$--form $\theta$ on $M$
\begin{equation}\label{for}
\nn_\dr\theta=-\frac{p}{r}\theta\quad\hbox{and}\quad\nn_X\theta=\n_X\theta-
\frac{1}{r}dr\wedge (i_X \theta),
\end{equation}

\begin{equation}\label{for2}
\nn_\dr dr=0\quad\hbox{and}\quad\nn_Xdr=rX^\flat.
\end{equation}

The curvature tensors $R$, $\bar R$ of $(M, g)$ and $(\bar M, \bar g)$ respectively,
are related by
\begin{equation}\label{cou}
\bar R(\dr,\cd)=0, \; \quad \bar R(X,Y)Z=R(X,Y)Z+g(X,Z)Y-g(Y,Z)X.
\end{equation}
From (\ref{cou}) one can see that $(\bar M, \bar g)$ is flat if and only if
$(M, g)$ has constant curvature 1.

From (\ref{for}) it is easily seen that we have the following relation
on the almost K\"ahler cone $(\bar M, \bar g, J, \omega)$
\begin{equation} \label{nnr} \nn_{\dr} \omega = 0.
\end{equation}
Since $(\bar M,J, \omega)$ is almost K{\"a}hler, the covariant derivative of $\omega$ satisfies
\begin{equation} \label{quasiK}
\nn_{J.} \omega(J., .) = \nn_{.} \omega(J.,J.) = -\nn_{.} \omega(.,.) .
\end{equation}
Using this and (\ref{nnr}), we also immediately have
\begin{equation} \label{nnxi}
\nn_{\xi} \omega = 0.
\end{equation}
This is equivalent to relation (\ref{nablaxiphi}).
In general, using (\ref{for}) if $X \in TM$
\begin{equation} \label{nnX}
\nn_{X} \omega = r^2( X^{\flat} \wedge \eta + \frac{1}{2} \n_X d\eta ),
\end{equation}
which implies the well known fact that $(M,g,\xi,\f,\eta)$ is Sasakian if and only if
$(\bar M,J, \omega)$  is K\"ahler.

We also collect some easy facts about the cone in the case when $(M,g,\xi,\f,\eta)$ is $K$-contact.
\begin{proposition} \label{Kcontobs}
Suppose $(M,g,\xi,\f,\eta)$ is $K$-contact. Then on the cone $(\bar M, \bar g, J, \omega)$, $\xi$ is an automorphism
of the almost K\"ahler structure. Moreover, we have
$$ \mbox{ (i)  } \; \; \nn_{A} \xi = -JA , \; \; \forall A \in T \bar M \; ;\hspace{14cm} $$
$$ \mbox{ (ii) } \; \; \bar R(A, \xi)B = - (\nn_A J)B \; \; \forall A, B \in T \bar M \; .\hspace{14cm} $$
\end{proposition}
\begin{proof} The fact that $\xi$ is an automorphism of the almost K\"ahler cone is clear.
Relation (i) follows straight from (\ref{nablaXxi}) in the $K$-contact case and the formulas
(\ref{vec}). As $\xi$ is Killing for $\bar g$, we have that
\begin{equation} \label{CurvKill}
 \bar R(A, \xi)B = \nn^2_{A,B} \xi = \nn_A \nn_B \xi - \nn_{\nn_A B} \xi \; .
\end{equation}
On the other hand, from (i)  $\nn^2_{A,B} \xi = - (\nn_A J)B $, so (ii) follows.
\end{proof}

A consequence of the previous proposition is the following:
\begin{proposition} \label{twoKcont}
Let $(M^{2n+1}, g)$ be a Riemannian manifold admitting two compatible $K$-contact structures
$(g, \phi, \xi, \eta)$, $(g, \widetilde{\phi}, \widetilde{\xi}, \widetilde{\eta})$. Denote by $(\bar g, J, \omega), \; (\bar g, \widetilde{J}, \widetilde{\omega})$ the corresponding
almost K\"ahler structures on the cone $\bar M$. Then for any $ A \in T \bar M $ we have
\begin{equation} \label{commutation}
[J, \widetilde{J}] A + \nn_A([\xi,\widetilde{\xi}]) = \bar R(\xi,\widetilde{\xi})A = - (\nn_{\xi} \widetilde{J})A = (\nn_{\widetilde{\xi}} J)A\; ,
\end{equation}
where $[J, \widetilde{J}]$ denotes the commutator of the two almost complex structures.
\end{proposition}
\begin{proof} Applying relation (\ref{CurvKill}) for $\xi$ and $\widetilde{\xi}$ we have
$$ \bar R(A, \xi) \widetilde{\xi} = \nn_A \nn_{ \widetilde{\xi}} \xi - \nn_{\nn_A \widetilde{\xi}} \xi \;  ; $$
$$ \bar R(A, \widetilde{\xi}) \xi = \nn_A \nn_{ \xi } \widetilde{\xi} - \nn_{\nn_A \xi } \widetilde{\xi} \;  . $$
Subtracting the two relations and using Bianchi's identity we get the first equality in (\ref{commutation}). The other equalities
in (\ref{commutation}) are consequences of (ii) of Proposition \ref{Kcontobs}.
\end{proof}


The next proposition deals with the case $ [\xi,\widetilde{\xi}] \not\equiv 0$ of our Sasakian--$K$-contact problem.
\begin{proposition} \label{SKnotcommuting}
Suppose $(M^{2n+1}, g)$ is a complete manifold admitting a Sasakian structure $(g, \phi, \xi, \eta)$
and a second compatible $K$-contact structure $(g, \widetilde{\phi}, \widetilde{\xi}, \widetilde{\eta})$
with $ [\xi,\widetilde{\xi}] \not\equiv 0$. Then $(M^{2n+1}, g)$ is covered by the round sphere $(S^{2n+1}, g_0)$, or
$n= 2k+1$ and $(M^{4k+3}, g)$ is 3-Sasakian. Moreover, the structure $(g, \widetilde{\phi}, \widetilde{\xi}, \widetilde{\eta})$ must be also
Sasakian.
\end{proposition}
\begin{proof} It is enough to consider the case $2n+1 \geq 5$. As $\widetilde{\xi}$ is Killing, it induces a
1-parameter group of isometries $\psi_t$. The pull-back by $\psi_t$ of the Sasakian structure $( g, \phi, {\eta},\xi )$ yields a family
of $g$-compatible Sasakian structures. By Theorem \ref{tyt} of Tachibana-Yu and Tanno, $(M^{2n+1}, g)$ must be isometric to a standard sphere $(S^{2n+1}, g_0)$
or is 3-Sasakian. That the structure $(g, \widetilde{\eta}, \widetilde{\xi})$ is also Sasakian follows in the first case from the result of Olszak \cite{Ol}
that in dimensions greater or equal to 5 a contact metric manifold of constant sectional curvature is necessarily Sasakian. If the manifold is 3-Sasakian
then it is Einstein and one can use the result of Boyer-Galicki \cite{bg} that any complete $K$-contact Einstein manifold is Sasakian.
\end{proof}

In view of this result, in the remaining of this section we will make the assumption
$ [\xi,\widetilde{\xi}] = 0$.


\begin{proposition} \label{twoKcontcomm}
Suppose $(M^{2n+1}, g)$ is a Riemannian manifold admitting two compatible $K$-contact structures
$(g, \phi, \xi, \eta)$, $(g, \widetilde{\phi}, \widetilde{\xi}, \widetilde{\eta})$ with $ [\xi,\widetilde{\xi}] = 0$. Then

\vspace{0.1cm}

(i) $\widetilde{\phi} \xi = \phi \widetilde{\xi}$.

\vspace{0.1cm}

(ii) The angle-function $f = g(\xi, \widetilde{\xi})$ has gradient given by
$\nabla f =  2 \phi \widetilde{\xi}$. Consequently, the only critical values of $f$ are $\pm 1$.
%
%
\end{proposition}

\begin{proof} Part (i) follows from
$$ 0 =  [\xi,\widetilde{\xi}] = \nabla_{\xi} \widetilde{\xi} - \nabla_{\widetilde{\xi}} \xi = - \phi \widetilde{\xi} + \widetilde{\phi} \xi ,$$
and part (ii) follows immediately from (i) and relation (\ref{gradfgen}).
%
%

\end{proof}

Next let us define some distributions that will be used in the rest of the argument.
Consider
$(M^{2n+1}, g)$ a Riemannian manifold admitting two compatible $K$-contact structures
$(g, \phi, \xi, \eta)$, $(g, \widetilde{\phi}, \widetilde{\xi}, \widetilde{\eta})$ with $ [\xi,\widetilde{\xi}] = 0$
as in Proposition \ref{twoKcontcomm}.  Denote by $U$ the (open, dense) set of regular points in $M$ of the angle function $f$, i.e.
$p \in U$ iff $\xi(p) \neq \pm \widetilde{\xi}(p)$.
On $U$ consider the 3-dimensional distribution
$\mathcal{V} = {\rm Span}(\xi, \widetilde{\xi}, \phi \widetilde{\xi})$ and let $ \mathcal{H} = \mathcal{V}^{\perp}$
the orthogonal complement with respect to the metric $g$. It is easily checked that
the ``horizontal'' distribution $ \mathcal{H} $ is both $\phi$ and $\widetilde{\phi}$-invariant.

Denote by $\bar U = U \x\RM^*_+ $ the corresponding subset
in the cone $\bar M$ and consider on $\bar U$ the distributions
$\mathcal{\bar V} = {\rm Span}(\xi, \widetilde{\xi}, \phi \widetilde{\xi}, \dr)$ and $\mathcal{H} = \mathcal{\bar V}^{\perp}$, using
a slight abuse of notation. The reader can check easily that $\mathcal{\bar V}$ and $\mathcal{H}$ are invariant with respect to both
$J$ and $ \widetilde{J}$. Moreover, restricted to $\mathcal{\bar V}$, $J$ and $\widetilde{J}$ commute.
In other words,  over $\bar U$ the distribution $\mathcal{\bar V}$ further splits in 2-dimensional distributions
$\mathcal{\bar V} = \mathcal{\bar V}^+ \oplus \mathcal{\bar V}^-$,
so that $J= \widetilde{J}$ on $\mathcal{\bar V}^+$ and $J= - \widetilde{J}$ on $\mathcal{\bar V}^-$.
In fact, one directly checks that
$$  \mathcal{\bar V}^+ = {\rm Span} ( \xi + \widetilde{\xi}, -(1 + f) r \dr + \phi \widetilde{\xi}) \; , \; \; \;
 \mathcal{\bar V}^- = {\rm Span} ( \xi - \widetilde{\xi}, -(1 - f) r \dr - \phi \widetilde{\xi}) \; . $$
Further, we have the following:
\begin{proposition} \label{lemmadistrib}
Let $(M^{2n+1}, g)$ be a Riemannian manifold admitting two compatible $K$-contact structures
$(g, \phi, \xi, \eta)$, $(g, \widetilde{\phi}, \widetilde{\xi}, \widetilde{\eta})$ with $ [\xi,\widetilde{\xi}] = 0$.
Assume further that $2n+1 = 5$, or that one of the structures is Sasakian. Then

(i)  $[J, \widetilde{J}] = 0$ on the cone $\bar M$;

(ii) The normalized gradient flow on $(M, g)$ of the angle function $f = g(\xi, \widetilde{\xi})$ consists of geodesics;

(iii) The distribution $\mathcal{H}$ is parallel along the gradient flow of $f$.
\end{proposition}

\begin{proof} (i) In the case one of the structures is Sasakian, the commutation of $J$ and $\widetilde{J}$ follows directly
from relation (\ref{commutation}). In the case $2n+1 = 5$, even without assuming that one of the structures is Sasakian,
note that the commutation of $J$ and $\widetilde{J}$ on $\mathcal{\bar V}$ holds and that $dim(\mathcal{H}) = 2$.
Thus, in this case, on $\mathcal{H}$, $J = \widetilde{J}$, or $J = - \widetilde{J}$, so $J$ and $\widetilde{J}$ commute everywhere.
Note also that in this case, virtue of relation (\ref{commutation}), we have
\begin{equation} \label{vertpar}
\bar \nabla_{V} J = \bar \nabla_{V} \widetilde{J} = 0, \; \; \forall V \in \mathcal{\bar V} \; .
\end{equation}
even if none of the structures is assumed Sasakian. This obviously also holds in all dimensions if one of the structures is assumed Sasakian.

(ii) Let $N=\frac{\phi\widetilde{\xi}}{\|\phi\widetilde{\xi}\|}$ be the normalized gradient vector field of the angle function $f$
on the regular set $U$. Using the Sasakian condition for one of the structures, one checks after a computation
that $\nabla_N N = 0$. Using relation (\ref{vertpar}), the same conclusion holds if $2n+1 = 5$ even without the Sasakian assumption.

(iii) Let $A$ be any section of the bundle $\mathcal{H}$. One has
$$
\begin{array} {l} <\nabla_NA,N>=-<A,\nabla_NN>=0\\
<\nabla_NA, \xi > =-<A,\nabla_N\xi > = <A,\phi N>=0\\
<\nabla_NA,\widetilde{\xi}> =-<A,\widetilde{\phi}N>=0 ,
\end{array}
$$
thus $\nabla_N A \in \mathcal{H}$.
\end{proof}

The following proposition is one more step toward the proof of Theorem \ref{codim2}. Together with Proposition \ref{SKnotcommuting}, it
already implies Theorem \ref{dim5}.
\begin{proposition} \label{morethandim5}
Suppose $(M^{2n+1}, g, \phi, \xi, \eta)$ is a Sasakian manifold admitting a second compatible $K$-contact structure
 $(g, \widetilde{\phi}, \widetilde{\xi}, \widetilde{\eta})$ with $ [\xi,\widetilde{\xi}] = 0$. Then $\phi$ and $\widetilde{\phi}$
 commute on $\mathcal{H} = {\rm Span} (\xi, \widetilde{\xi}, \phi \widetilde{\xi})^{\perp}$.
If, additionally, we have that
 $\phi = \widetilde{\phi}$ on $\mathcal{H}$, or $\phi = - \widetilde{\phi}$ on $\mathcal{H}$,
 then $(g, \widetilde{\phi}, \widetilde{\xi}, \widetilde{\eta})$ is also Sasakian.
 \end{proposition}

 \begin{proof} The commutation of $\phi$ and $\widetilde{\phi}$
 on $\mathcal{H}$ certainly follows from the fact that $J$ and $\widetilde{J}$ commute on the entire cone $\bar M$.
 The assumption $\phi = \widetilde{\phi}$ on $\mathcal{H}$, or $\phi = - \widetilde{\phi}$ on $\mathcal{H}$
 translates in $J = \widetilde{J}$ or $J = - \widetilde{J}$ on $\mathcal{H}$.
 Consider the case $J= \widetilde{J}$ on the whole $\mathcal{\bar H}$ (the other case follows from this by replacing $J$ by $-J$).
Then a simple computation shows
$$ \widetilde{\omega} = \omega - \frac{1}{2(1 - f)} (\widetilde{\eta} - \eta) \wedge J (\widetilde{\eta} - \eta) =
\omega - \frac{1}{2(1 - f)} (\widetilde{\eta} - \eta) \wedge d(r^2 (1 - f)) \; .$$

For any vector field $H \in \mathcal{\bar H}$, we have $H(f) = 0$, but also
$$ \bar \nabla_H (\widetilde{\eta} - \eta) = 0 \; , \; \;  \bar \nabla_H (J(\widetilde{\eta} - \eta)) = 0 . $$
The first equality follows because
$ \bar \nabla_H (\widetilde{\xi} - \xi) = - \widetilde{J} H + JH = 0 $. The second equality follows from the first
and using the assumption that $(g, \phi, \xi)$ is Sasakian, hence $\bar \nabla J = 0$.
Thus, $\bar \nabla_H \widetilde{\omega} = 0$ and combining this with (\ref{vertpar}), it follows
that $\bar \nabla_X \widetilde{\omega} = 0$ for any vector $X$. This is verified first on the set $U$ where $\xi(p) \neq \widetilde{\xi}(p)$, but then everywhere
by $M$ by the density of the set $U$.
\end{proof}

{\it Proof of Theorem \ref{dim5}.} The case $ [\xi,\widetilde{\xi}] \not\equiv 0$ is covered by Proposition \ref{SKnotcommuting} (certainly, the 3-Sasakian
case does not occur in dimension 5). The case $ [\xi,\widetilde{\xi}] = 0$, is solved by Proposition \ref{morethandim5} simply noting that in dimension
$2n+1 = 5$ the condition  $\phi = \widetilde{\phi}$, or $\phi = - \widetilde{\phi}$ on $\mathcal{H}$ is automatically satisfied, as the distribution
$\mathcal{H}$ is 2-dimensional, and is $\phi$ and $\widetilde{\phi}$  invariant. $\Box$

\vspace{0.2cm}

For the proof of Theorem \ref{codim2} we need two more steps. The first step clarifies the organization of the critical set of the angle function.

\begin{proposition} Let $(M^{2n+1}, g, \phi, \eta , \xi)$ be a compact $K$-contact manifold. Let $\widetilde{\xi}$
be a unit Killing vector field with $[\xi ,\widetilde{\xi}]=0$ on $M$
and assume the angle function $f = g(\widetilde{\xi},\xi )$ is not a constant function. Then $f$ has exactly two critical manifolds, $\Sigma_{-1}$, $\Sigma_1$,
the $-1$ and $1$ level sets.
\end{proposition}

\begin{proof} The function $f=g(\widetilde{\xi},\xi )$ is $\widetilde{\xi}$ invariant (it is also $\xi$ invariant),
its gradient vector field is $2\phi \widetilde{\xi}$ and, therefore, a point $p$ is critical for $g(\widetilde{\xi},\xi )$ if and only if $\phi \widetilde{\xi}=0$ at $p$,
that is $\widetilde{\xi}=\pm\xi $ at $p$. Each connected component of the critical set is a closed, totally geodesic, contact invariant sub-manifold of $M$ as was shown in \cite{RUK}. Since they are all non-degenerate critical sub-manifolds of even index (see \cite{RUK}),
there is exactly one component at the minimum level and exactly one component at the maximum level of $g(\widetilde{\xi},\xi )$.
\end{proof}
%
Next,
let $\Sigma_{-1}$, $\Sigma_0$, $\Sigma_1$ denote, respectively, the level $-1$, $0$ and $1$ manifolds of the angle function
$f =g(\xi ,\widetilde{\xi} )$. Through any point $p\in \Sigma_0$, there passes a unique geodesic $\gamma$,
intersecting orthogonally each of the three sub-manifolds. Moreover, the unit tangent vector of $\gamma$
is $N=\frac{\phi \widetilde{\xi}}{\|\phi \widetilde{\xi}\|}$ away from the critical set of $f$.

\begin{proposition} \label{laststep}  With the assumptions from Theorem \ref{codim2},
suppose that one of the critical manifolds $\Sigma_{-1}$ and $\Sigma_1$ has codimension 2. Then
either $\phi = \widetilde{\phi}$ on $\mathcal{H}$ or $\phi = - \widetilde{\phi}$ on $\mathcal{H}$.
\end{proposition}

\begin{proof} Suppose $\Sigma_1$ has codimension 2 and let $p=\gamma (0)\in \Sigma_0$,
$q=\gamma (T)\in \Sigma_1$.
Then the parallel transport of $\mathcal{H}$ from $p$ to $q$
along $\gamma$ is orthogonal to $N$ and $\phi N$ and therefore is
fully contained in $T_q\Sigma_1$. But in $T\Sigma_1$, one has $\phi \widetilde{\phi} =-1$ proving that
all eigenvalues of $\phi \widetilde{\phi}$ are all equal to $-1$.
If one assumes $\Sigma_{-1}$ to have codimension 2, the result will be that all eigenvalues of $\phi \widetilde{\phi}$ are equal to $1$.
\end{proof}

\vspace{0.2cm}

To conclude, here is the proof of Theorem \ref{codim2}. Combining Propositions \ref{laststep} and \ref{morethandim5}, it follows that the structure
$(g, \widetilde{\xi}, \widetilde{\eta})$ is also Sasakian. Now the conclusion follows from Theorem \ref{tyt}, noting that, by assumption,
the angle function is not constant.

\vspace{0.2cm} \noindent {\bf Acknowledgments:} The authors are
grateful to David Blair and the referee for useful
comments about this note.

\label{lastpage-01}
\end{document}